\documentclass[11pt]{amsart}
\usepackage{fullpage}
\usepackage[latin1]{inputenc}
\usepackage[T1]{fontenc}
\usepackage{amsfonts}
\usepackage{amsmath}
\usepackage{amssymb}
\usepackage{graphicx}
\usepackage{pstricks}
\usepackage{pst-grad}
\usepackage{multido}
\usepackage{amsthm}
\usepackage{enumerate}
\usepackage[absolute]{textpos}
\definecolor{purple}{rgb}{0.8,0.12,0.8}
\definecolor{orange}{rgb}{1.0,0.7,0.0}
\definecolor{pink}{rgb}{1,0.5,0.8}
\definecolor{blackg}{rgb}{0.1,0.25,0.1}
\definecolor{ForestGreen}{cmyk}{0.91,0,0.88,0.42}
\definecolor{Turquoise}{cmyk}{0.85,0,0.20,0}

\newcommand{\blambda}{\boldsymbol{\lambda}}
\newcommand{\bnu}{\boldsymbol{\nu}}
\newcommand{\bmu}{\boldsymbol{\mu}}

\newcommand{\el}{\lambda}

\newtheorem{Th}{Theorem}[section]

\newtheorem{Lem}[Th]{Lemma}

\newtheorem{Prop}[Th]{Proposition}
\newtheorem{Def-Prop}[Th]{Definition-Proposition}

\theoremstyle{definition}
\newtheorem{Def}[Th]{Definition}

\theoremstyle{remark}
\newtheorem{Rem}[Th]{Remark}
\newtheorem{abs}[Th]{\bfseries}{\rmfamily}
\title{Schur elements for the Ariki-Koike algebra and applications}
\author{Maria Chlouveraki and Nicolas Jacon}
\thanks{The authors would like to thank I.~Gordon, S.~Griffeth, M.~Fayers and A.~Mathas for useful conversations.  In particular, the first author is indebted to Stephen Griffeth for explaining the results of his paper \cite{DG}, which inspired part of this paper. Maria Chlouveraki gratefully acknowledges the support of the EPSRC through the grant EP/G04984X/1. }
\begin{document}

\begin{abstract}
 We study the Schur elements associated to the simple modules of the Ariki-Koike algebra. We first give a cancellation-free formula for  them so that their factors can be easily read and programmed.  
  We then study direct applications of this result.  We also complete the determination of the canonical basic sets for cyclotomic Hecke algebras of type $G(l,p,n)$ in characteristic $0$.
\end{abstract}

\maketitle
\begin{section}{Introduction}

Schur elements play a powerful role in the representation theory of symmetric algebras.
 In the case of the Ariki-Koike algebra, that is, the Hecke algebra of the complex reflection group $G(l,1,n)$, 
   they  are Laurent polynomials whose factors determine when Specht modules are projective irreducible and whether the algebra is semisimple.
 
Formulas for the Schur elements of the Ariki-Koike algebra have been obtained independently, first by  
 Geck, Iancu and Malle \cite{GIM}, and later by Mathas \cite{Mat}.  The first  aim of this paper  is to give a cancellation-free formula for these polynomials (Theorem \ref{canfreeform}), so that their factors can be easily read and programmed.
 We then present  a number of direct applications.  These include a new formula for Lusztig's $a$-function, as well as a simple classification of the projective irreducible modules for Ariki-Koike algebras (that is, the blocks of defect $0$). 

The second part of the paper is devoted to another aspect of the representation theory of these algebras  in connection with these Schur elements: the theory of canonical basic sets. 
 The main aim here  is 
to obtain a classification of the simple modules for specialisations of cyclotomic Hecke algebras in characteristic $0$.  
 In \cite{CJ1}, we studied  mainly the case of finite Weyl groups. In this paper, we focus on cyclotomic Hecke algebras of type $G(l,p,n)$.  
 Using Lusztig's  $a$-function, defined from the Schur elements,    
 the theory of canonical basic sets provides a natural and efficient  way to parametrise the simple modules of these algebras. 

  The existence and  explicit  determination of the canonical basic sets is already known in the case of Hecke algebras of finite Weyl groups 
   (see \cite{book} and \cite{CJ1}). The case of cyclotomic Hecke algebras of type $G(l,p,n)$ has been partially studied in \cite{book,GenJa}, and recently in \cite{CGG}
     using the theory of Cherednik algebras. Answering a question raised in \cite{CGG}, the goal of the last part of this paper is to complete the determination of the canonical basic sets in this case. 
%
%

\end{section}

\section{Preliminaries}
In this section, we introduce the necessary  definitions and notation. 
\begin{abs} 
A {partition} $\lambda=(\lambda_1,\lambda_2,\lambda_3,\ldots)$ is a decreasing sequence of non-negative integers. We define the {length of} $\lambda$ to be the smallest integer $\ell(\lambda)$ such that $\lambda_i=0$ for all $i>\ell(\lambda)$.
We write $|\lambda|:=\sum_{i \geq 1}\lambda_i$ and we say that $\lambda$ is a {partition of} $m$, for some $m \in \mathbb{N}$, if $m=|\lambda|$. 
We set $n(\lambda):=\sum_{i \geq 1}  (i-1)\lambda_i$.

We define the set of nodes $[\lambda]$ of $\lambda$ to be the set
$$[\lambda]:=\{(i,j)\,\,|\,\, i\geq 1,\,\,1 \leq j \leq \lambda_i\}.$$
A node $x=(i,j)$ is called {removable} if $[\lambda] \setminus \{(i,j)\}$ is still the set of nodes of a partition. Note that if $(i,j)$ is removable, then $j=\lambda_i$.

The {conjugate partition} of $\lambda$ is the partition $\lambda'$ defined by
$$\lambda'_{k}:=\#\{i\,|\,i\geq 1 \text{ such that } \lambda_i\geq k\}.$$
Obviously, $\lambda_1'=\ell(\lambda)$. 
The set  of nodes of $\lambda'$ satisfies
$$(i,j) \in [\lambda'] \Leftrightarrow (j,i)\in [\lambda].$$
Note that if $(i,\lambda_i)$ is a removable node of $\lambda$, then
$\lambda_{\lambda_i}'=i.$
Moreover, we have $$n(\lambda)=\sum_{i \geq 1}  (i-1)\lambda_i= \frac{1}{2}\sum_{i \geq 1}(\lambda'_i-1)\lambda'_i.$$ 
If $x=(i,j) \in [\lambda]$ and $\mu$ is another partition, we define the {\em generalised hook length of $x$ with respect to }($\lambda$, $\mu$) to be the integer:
            $$h_{i,j}^{\lambda,\mu}:=\lambda_i-i+\mu'_j-j+1.$$ 
For $\mu=\lambda$, the above formula becomes the classical hook length formula (giving the length of the hook of $\lambda$ that $x$ belongs to).
Moreover, we define the {content} of $x$ to be the difference
$$\mathrm{cont}(x)=j-i.$$
The following lemma, whose proof is an easy combinatorial exercise (with the use of Young diagrams), relates the contents of the nodes of (the ``right rim'' of) $\lambda$ with the contents of the nodes of (the ``lower rim'' of) $\lambda'$.
\end{abs}
\begin{Lem}\label{conj cont}
Let $\lambda=(\lambda_1,\lambda_2,\ldots)$ be a partition and let $k$ be an integer such that $1 \leq k \leq \lambda_1$. Let $q$ and $y$ be two indeterminates. Then we have
 $$\frac{1}{(q^{\lambda_1}y-1)}\cdot\left(\prod_{1\leq i\leq \lambda_k'} \frac{q^{\lambda_i-i+1}y-1}{q^{\lambda_i-i}y-1} \right)=\frac{1}{(q^{-\lambda'_{k}+k-1}y-1)}\cdot\left(\prod_{k \leq j\leq \lambda_1} \frac{q^{-\lambda'_j+j-1}y-1}{q^{-\lambda'_j+j}y-1} \right).
$$
\end{Lem}

\begin{abs}
Let $l$ and $n$ be positive integers.
 An $l$-{partition} of  $n$ is an ordered $l$-tuple $\blambda=(\lambda^{0},\lambda^{1},\ldots,\lambda^{l-1})$  of partitions  such that $\sum_{0 \leq s \leq l-1}|\lambda^{s}|=n$. We denote by $\Pi^l_n $ the set of $l$-partitions of $n$.
\end{abs}

\begin{abs}
Let $R$ be a commutative domain with $1$. Fix elements
$q,\,Q_0,\,\ldots,\,Q_{l-1}$ of $R$, and assume that $q$ is invertible in $R$. Set ${\bf q}:=(Q_0,\,\ldots,\,Q_{l-1}\,;\,q)$.
The {\em Ariki-Koike algebra} $\mathcal{H}^{\bf q}_{n}$ is the unital associative $R$-algebra with generators $T_0,\,T_1,\,\ldots,\,T_{n-1}$ and relations:
$$\begin{array}{rl}
(T_0 -Q_0) (T_0 -Q_1)\cdots(T_0 -Q_{l-1})=0,& \\
(T_i-q)(T_i+1)=0  & \text{for $1\leq i \leq n-1$},\\
T_0T_1T_0T_1=T_1T_0T_1T_0&,\\
T_iT_{i+1}T_i=T_{i+1}T_iT_{i+1} & \text{for $1\leq i \leq n-2$},\\
T_iT_j=T_jT_i  &\text{for  $0\leq i <j \leq n-1$ with $j-i>1$.}
\end{array}$$
The last three relations are the {\em braid relations} satisfied by $T_0,\,T_1,\,\ldots,\,T_{n-1}$.

The Ariki-Koike algebra $\mathcal{H}^{\bf q}_{n}$ is a deformation of the group algebra of the complex reflection group $G(l,1,n)=(\mathbb{Z}/l\mathbb{Z})\wr \mathfrak{S}_n$. Ariki and Koike \cite{ArKo} have proved  that  $\mathcal{H}^{\bf q}_{n}$  is a free $R$-module of rank $l^n n!=|G(l,1,n)|$ (see \cite[Proposition 13.11]{arikilivre}). 
In addition, when $R$ is a field, 
they have  constructed a simple $\mathcal{H}^{\bf q}_{n}$-module $V^{\blambda}$, with character $\chi^{\blambda}$, for each $l$-partition $\blambda$ of $n$ (see \cite[Theorem 13.6]{arikilivre}). These modules form a complete set of non-isomorphic simple modules
 in the case where $\mathcal{H}^{\bf q}_{n}$ is split semisimple  (see \cite[Corollary 13.9]{arikilivre}).

\end{abs}

\begin{abs} There is a useful  semisimplicty criterion for Ariki-Koike algebras which has been given by Ariki in \cite{Ar}. This criterion will be recovered from our results later  (see Theorem \ref{ssimple}), so
 let us simply assume from now on that   $\mathcal{H}^{\bf q}_{n}$ is split semisimple. This happens, for example,  when $q, \,Q_0,\,\ldots,\,Q_{l-1}$ are indeterminates and  $R=\mathbb{Q}(q,Q_0,\,\ldots,\,Q_{l-1})$. 

 Now, there exists a linear form $\tau:\mathcal{H}^l_{n} 
\rightarrow R$ which was introduced by Bremke and Malle in \cite{BreMa}, and was proved to be symmetrizing by Malle and Mathas in \cite{MaMa} whenever all $Q_i$'s are invertible in $R$. 
An explicit description of this form can be found in any of these two articles. Following Geck's results on symmetrizing forms (see \cite[Theorem 7.2.6]{GePf}), we obtain the following definition for the Schur elements associated to the irreducible representations of $\mathcal{H}^{\bf q}_{n}$.
\end{abs}
\begin{Def}\label{Schur}
Suppose that $R$ is a field and that $\mathcal{H}^{\bf q}_{n}$ is split semisimple. The {Schur elements} of 
$\mathcal{H}^{\bf q}_{n}$ are the elements $s_{\blambda}({\bf q})$ of $R$ such that
$$\tau = \sum_{\blambda \in \Pi_n^l} \frac{1}{s_{\blambda}({\bf q})}\chi^{\blambda}. $$
\end{Def}




\begin{abs}
The Schur elements of the Ariki-Koike algebra $\mathcal{H}^{\bf q}_{n}$ have been independently calculated by Geck, Iancu and Malle \cite{GIM}, and by Mathas \cite{Mat}. 
From now on, for all $m \in \mathbb{N}$, let $[m]_q:=(q^m-1)/(q-1)=q^{m-1}+q^{m-2}+\cdots+q+1$.
The formula given by Mathas does not demand extra notation and is the following:

\begin{Th}\label{Mathas} Let $\blambda=(\lambda^{0},\lambda^{1},\ldots,\lambda^{l-1})$ be an $l$-partition of $n$. Then
$$s_{\blambda}({\bf q})=(-1)^{n(l-1)} (Q_0Q_1\cdots Q_{l-1})^{-n}
q^{-\alpha({\blambda}')} \prod_{0\leq s \leq l-1}  \prod_{(i,j) \in [\lambda^{s}]}
 Q_s[{{h_{i,j}^{\lambda^s,\lambda^{s}}}}]_q \,\cdot
 \prod_{0 \leq s< t \leq l-1} X_{st}^{\blambda} ,$$
where
$$\alpha({\blambda}')=\frac{1}{2}\sum_{0\leq s \leq l-1}\sum_{i \geq 1}
(\lambda^{s'}_i-1)\lambda^{s'}_i$$
and
$$ X_{st}^{\blambda}=  
 \prod_{(i,j) \in [\lambda^{t}]}(q^{j-i}Q_t-Q_s) 
\cdot
 \prod_{(i,j) \in [\lambda^{s}]}\left((q^{j-i}Q_s-q^{\lambda_1^{t}}Q_t) 
\prod_{1\leq k\leq \lambda_1^{t}}
\frac{q^{j-i}Q_s-q^{k-1-\lambda_k^{{t}'}}Q_t}{q^{j-i}Q_s-q^{k-\lambda_k^{t'}}Q_t}\right).  
$$
\end{Th}
$ $

The formula by Geck, Iancu and Malle is more symmetric, and describes the Schur elements in terms of {\em beta numbers}. If $\blambda=(\lambda^{0},\lambda^{1},\ldots,\lambda^{l-1})$ is an $l$-partition of $n$, then the {\em length of} $\blambda$ is $\ell(\blambda)=\mathrm{max}\{\ell(\lambda^s)\,|\,0\leq s \leq l-1\}$. 
 Fix an integer $L$ such that $L \geq \ell(\blambda)$. The $L$-{\em beta numbers} for $\lambda^{s}$ are the integers $\beta_i^{s}=\lambda_i^{s}+L-i$ for $i = 1, \ldots , L$. 
Set $B^{s} = \{\beta_1^{s}, \ldots, \beta_L^{s}\}$ for $s=0,\ldots,l-1$. The matrix ${\bf B}=(B^{s})_{0\leq s \leq l-1}$ is called the $L$-{symbol} of $\blambda$.

\begin{Th}\label{sym}
Let  $\blambda=(\lambda^{0},\ldots,\lambda^{l-1})$ be an $l$-partition of $n$
with $L$-symbol ${\bf B}=(B^{s})_{0\leq s \leq l-1}$, where $L \geq \ell(\blambda)$.  Let $a_{L}:=n(l-1)+\binom{ l}{ 2}\binom{ L}{ 2}$ and $b_{L}:=lL(L-1)(2lL-l-3)/12$. Then  $$s_{\blambda}({\bf q})=(-1)^{a_{L}} q^{b_{L}}(q-1)^{-n}(Q_0Q_1\ldots Q_{l-1})^{-n}\nu_{\blambda}/ \delta_{\blambda},$$ where
$$
\nu_{\blambda}=
\prod_{0\leq s<t\leq l-1}(Q_s-Q_t)^L\prod_{0 \leq s,\,t \leq l-1}\prod_{b_s \in B^{s}}\prod_{1 \leq k \leq b_s} 
(q^kQ_s-Q_t)$$
and
$$\delta_{\blambda}=\prod_{0\leq s< t \leq l-1}\prod_{(b_s,b_t) \in B^{s}\times B^{t}}(q^{b_s}Q_s-q^{b_t}Q_t) \prod_{0 \leq s \leq l-1} \prod_{1 \leq i < j \leq L}(q^{\beta_i^{s}}Q_s-q^{\beta_j^{s}}Q_s).
$$
\end{Th}
$ $
As the reader may see, in both formulas above, the factors of $s_{\blambda}({\bf q})$ are not obvious.
Hence, it is not obvious for which values of ${\bf q}$ the Schur element $s_{\blambda}({\bf q})$ becomes zero.
\end{abs}
\section{A cancellation-free formula for the Schur elements}
In this section, we will give a  cancellation-free formula for the Schur elements of $\mathcal{H}^{\bf q}_{n}$. This formula is also symmetric.

\begin{abs}
Let $\blambda=(\lambda^{0},\lambda^{1},\ldots,\lambda^{l-1})$ be an $l$-partition of $n$. 
The multiset $(\lambda_i^{s})_{0 \leq s \leq l-1,\,i\geq 1}$ is a composition of $n$ ({\em i.e.}
a multiset of non-negative integers whose sum is equal to $n$). By reordering the elements of this composition, we obtain a partition of $n$. We denote this partition by $\bar{\blambda}$.
(e.g., if $\blambda=((4,1),\emptyset,(2,1))$, then $\bar{\blambda}=(4,2,1,1)$).
\end{abs}

\begin{Th}\label{canfreeform} Let $\blambda=(\lambda^{0},\lambda^{1},\ldots,\lambda^{l-1})$ be an $l$-partition of $n$. 
Then 
\begin{equation}\label{pretty}
s_{\blambda}({\bf q})=(-1)^{n(l-1)}q^{-n(\bar{\blambda})}(q-1)^{-n} \prod_{0\leq s \leq l-1}
\prod_{(i,j) \in [\lambda^s]} \prod_{0\leq t \leq l-1}  (q^{h_{i,j}^{\lambda^s,\lambda^{t}}}Q_sQ_t^{-1}-1).
\end{equation}
Since the total number of nodes in $\blambda$ is equal to $n$, the above formula can be rewritten as follows:
\begin{equation}\label{claim}
s_\lambda({\bf q})=(-1)^{n(l-1)}q^{-n(\bar{\blambda})}
 \prod_{0 \leq s \leq l-1} \prod_{(i,j) \in [\lambda^{s}]}\left( [{{h_{i,j}^{\lambda^s,\lambda^{s}}}}]_q
 \prod_{0 \leq t \leq l-1,\, t\neq s} (q^{h_{i,j}^{\lambda^s,\lambda^{t}}}Q_sQ_t^{-1}-1)\right).
 \end{equation}
\end{Th}
\begin{abs}

We will now proceed to the proof of the above result.
Following Theorem \ref{Mathas}, we have that

$$s_{\blambda}({\bf q})=(-1)^{n(l-1)} (Q_0Q_1\cdots Q_{l-1})^{-n}
q^{-\alpha({\blambda}')} \prod_{0\leq s \leq l-1}  \prod_{(i,j) \in [\lambda^{s}]}
 Q_s[{{h_{i,j}^{\lambda^s,\lambda^{s}}}}]_q \,\cdot
 \prod_{0 \leq s< t \leq l-1} X_{st}^{\blambda} ,$$
where
$$\alpha({\blambda}')=\frac{1}{2}\sum_{0 \leq s \leq l-1}\sum_{i \geq 1}
(\lambda^{s'}_i-1)\lambda^{s'}_i$$
and
$$ X_{st}^{\blambda}=  
 \prod_{(i,j) \in [\lambda^t]}(q^{j-i}Q_t-Q_s) 
\cdot
 \prod_{(i,j) \in [\lambda^s]}\left((q^{j-i}Q_s-q^{\lambda_1^t}Q_t) 
\prod_{1\leq k \leq \lambda_1^t}
\frac{q^{j-i}Q_s-q^{k-1-\lambda_k^{t'}}Q_t}{q^{j-i}Q_s-q^{k-\lambda_k^{t'}}Q_t}\right).  
$$

The following lemma relates the terms $q^{-n(\bar{\blambda})}$ and $q^{-\alpha({\blambda}')}$ .
\end{abs}
\begin{Lem}\label{lemma q}
Let $\blambda$ be an $l$-partition of $n$. We have that
$$\alpha(\blambda')+\sum_{0\leq s <t\leq l-1}\sum_{i\geq 1}\lambda^{s'}_i\lambda^{t'}_i=
n(\bar{\blambda}).$$
\end{Lem}
\begin{proof}
Following the definition of the conjugate partition, we have
$\bar{\blambda}^{'}_i = \sum_{0\leq s \leq l-1}\lambda_i^{s'},$
for all $i \geq 1$. Therefore,
$$\begin{array}{rcl}
n(\bar{\blambda})&=&\displaystyle{\frac{1}{2}\sum_{i \geq 1}(\bar{\blambda}^{'}_i-1)\bar{\blambda}^{'}_i}\\ \smallbreak
  &=& 
\displaystyle\frac{1}{2}\sum_{i \geq 1}\left(\left(\sum_{0\leq s \leq l-1}\lambda_i^{s'}-1\right)\cdot\sum_{0\leq s \leq l-1}\lambda_i^{s'}\right)\\ \smallbreak
&=&\displaystyle\frac{1}{2}\sum_{i \geq 1}
\left( \sum_{0\leq s <t\leq l-1}2\cdot\lambda^{s'}_i\lambda^{t'}_i 
+\sum_{0\leq s \leq l-1}{\lambda_i^{s'}}^2-\sum_{0\leq s \leq l-1}{\lambda_i^{s'}}\right)\\ \smallbreak
&=& \displaystyle\sum_{0\leq s <t\leq l-1}\sum_{i \geq 1}\lambda^{s'}_i\lambda^{t'}_i +
\frac{1}{2} \sum_{0\leq s \leq l-1}\sum_{i \geq 1}(\lambda^{s'}_i-1)\lambda^{s'}_i\\ \smallbreak
&=&\displaystyle\sum_{0\leq s <t\leq l-1}\sum_{i \geq 1}\lambda^{s'}_i\lambda^{t'}_i +\alpha(\blambda')\end{array}
$$

\end{proof}

Hence, to prove Equality (\ref{claim}), it is enough to show that, for all $0\leq s <t\leq l-1$,
\begin{equation}\label{X}
X_{st}^{\blambda}=q^{-\sum_{i\geq 1}\lambda^{s'}_i\lambda^{t'}_i} 
Q_s^{| \lambda^{t}|} Q_t^{|\lambda^{s}|}
\prod_{(i,j) \in [\lambda^{s}]}
 (q^{h_{i,j}^{\lambda^s,\lambda^{t}}}Q_sQ_t^{-1}-1)\cdot 
\prod_{(i,j) \in [\lambda^{t}]}
 (q^{h_{i,j}^{\lambda^t,\lambda^s}}Q_tQ_s^{-1}-1). 
\end{equation}

\begin{abs}
We will proceed by induction on the number of nodes of $\lambda^{s}$. We do not need to do the same for $\lambda^{t}$, because the symmetric formula for the Schur elements given by Theorem \ref{sym} implies the following: if $\bmu$ is the multipartition obtained from $\lambda$ by exchanging 
$\lambda^{s}$ and $\lambda^t$, then
$$X_{st}^{\blambda}(Q_s,Q_t)=X_{st}^{\bmu}(Q_t,Q_s).$$
If $\lambda^{s}=\emptyset$, then
$$\begin{array}{rcl}
X_{st}^{\blambda}&=&\displaystyle \prod_{(i,j) \in [\lambda^{t}]}(q^{j-i}Q_t-Q_s)\\ &=&\displaystyle Q_s^{| \lambda^{t}|}
\prod_{(i,j) \in [\lambda^{t}]}(q^{j-i}Q_tQ_s^{-1}-1)\\
&=&
\displaystyle Q_s^{| \lambda^{t}|}\prod_{1\leq i\leq \lambda^{t'}_1} \prod_{1\leq j \leq \lambda^{t}_i}
(q^{j-i}Q_tQ_s^{-1}-1)\\
 &=&
\displaystyle Q_s^{| \lambda^{t}|}\prod_{1 \leq i \leq \lambda^{t'}_1} \prod_{1\leq j \leq \lambda^{t}_i}
(q^{\lambda^{t}_i-j+1-i}Q_tQ_s^{-1}-1)\\
&=&
\displaystyle Q_s^{| \lambda^{t}|}\prod_{(i,j) \in [\lambda^{t}]}
 (q^{h_{i,j}^{\lambda^t,\lambda^{s}}}Q_tQ_s^{-1}-1),\end{array}
$$
as required.
\end{abs}
\begin{abs}
Now assume that our assertion holds when $\#[\lambda^{s}] \in \{0,1,2,\ldots,N-1\}$. We want to show that it also holds when $\#[\lambda^{s}]=N \geq 1$.
If $\lambda^{s}\neq \emptyset$, then there exists $i$ such that $(i,\lambda_i^{s})$ is a removable node of $\lambda^{s}$. Let $\bnu$ be the multipartition defined by
$$\nu^{s}_i:=\lambda^{s}_i-1,\,\,\, \nu^s_j:=\lambda^s_j \text{ for all } j\neq i,\,\,\, \nu^t:=\lambda^t \text{ for all } t\neq s.$$
Then $[\lambda^s]=[\nu^s] \cup \{(i,\lambda^s_i)\}$. Since Equation $(\ref{X})$ holds for
$X_{st}^{\bnu}$ and 
$$X_{st}^{\blambda}=X_{st}^{\bnu} \cdot \left((q^{\lambda^s_i-i}Q_s-q^{\lambda_1^t}Q_t) 
\prod_{1\leq k \leq \lambda_1^t}
\frac{q^{\lambda^s_i-i}Q_s-q^{k-1-\lambda_k^{t'}}Q_t}{q^{\lambda^s_i-i}Q_s-q^{k-\lambda_k^{t'}}Q_t}\right),$$
it is enough to show that (to simplify notation, from now on set $\lambda:=\lambda^s$ and $\mu:=\lambda^t$):
\begin{equation}\label{add1}
(q^{\lambda_i-i}Q_s-q^{\mu_1}Q_t)
\prod_{1\leq k \leq \mu_1}
\frac{q^{\lambda_i-i}Q_s-q^{k-1-\mu_k'}Q_t}{q^{\lambda_i-i}Q_s-q^{k-\mu_k'}Q_t}
=
q^{-\mu_{\lambda_i}'}
Q_t  (q^{\lambda_i-i+\mu_{\lambda_i}'-\lambda_i+1}Q_sQ_t^{-1}-1)\cdot A \cdot B,
\end{equation}
where
$$A:=
\prod_{1\leq k\leq \lambda_i-1}
\frac{q^{\lambda_i-i+\mu_k'-k+1}Q_sQ_t^{-1}-1}
{q^{\lambda_i-i+\mu_k'-k}Q_sQ_t^{-1}-1}$$
and
$$B:=\prod_{1\leq k \leq \mu_{\lambda_i}'}
\frac{q^{\mu_k-k+\lambda_{\lambda_i}'-\lambda_i+1}Q_tQ_s^{-1}-1}
{q^{\mu_k-k+\lambda_{\lambda_i}'-\lambda_i}Q_tQ_s^{-1}-1}.$$
\\
Note that, since $(i,\lambda_i)$ is a removable node of $\lambda$, we have $\lambda_{\lambda_i}'=i$.
We have that
$$A=q^{\lambda_i-1}
\prod_{1 \leq k \leq \lambda_i-1}
\frac{q^{\lambda_i-i}Q_s-q^{k-1-\mu_k'}Q_t}
{q^{\lambda_i-i}Q_s-q^{k-\mu_k'}Q_t}.$$
Moreover, by Lemma \ref{conj cont}, for $y=q^{i-\lambda_i}Q_tQ_s^{-1},$ we obtain that
$$B=\frac{(q^{\mu_1+i-\lambda_i}Q_tQ_s^{-1}-1)}
{(q^{-\mu_{\lambda_i}'+\lambda_i-1+i-\lambda_i}Q_tQ_s^{-1}-1)}\cdot
\left(\prod_{\lambda_i\leq k \leq \mu_{1}}
\frac{q^{-\mu_k'+k-1+i-\lambda_i}Q_tQ_s^{-1}-1}
{q^{-\mu_k'+k+i-\lambda_i}Q_tQ_s^{-1}-1}\right),$$
{i.e.,}
$$B=Q_t^{-1}q^{\mu_{\lambda_i}'-\lambda_i+1}
\frac{(q^{\lambda_i-i}Q_s-q^{\mu_1}Q_t)}
{(q^{\mu_{\lambda_i}'-\lambda_i+1+\lambda_i-i}Q_sQ_t^{-1}-1)}\cdot
\left(\prod_{\lambda_i\leq k \leq \mu_{1}}
\frac{q^{\lambda_i-i}Q_s-q^{k-1-\mu_k'}Q_t}{q^{\lambda_i-i}Q_s-q^{k-\mu_k'}Q_t}\right).$$
\\
Hence, Equality (\ref{add1}) holds.
\end{abs}

\begin{section}{First consequences}
We give here several direct applications of Formula (\ref{claim}) obtained in Theorem \ref{canfreeform}.

\begin{abs}
A first  application of Formula (\ref{claim}) is that we can easily recover a
well-known semisimplicity criterion  for the Ariki-Koike algebra due to Ariki \cite{Ar}. To do this, 
 let us assume that $q,\,Q_0,\,\ldots,\,Q_{l-1}$ are indeterminates and $R=\mathbb{Q}(q,Q_0,\,\ldots,\,Q_{l-1})$. Then the resulting  ``generic'' Ariki-Koike algebra  
$  \mathcal{H}^{\bf q}_{n}$ is split semisimple. Now
    assume  that $\theta :  \mathbb{Z}[q^{\pm 1},Q_0^{\pm 1},\,\ldots,\,Q_{l-1}^{\pm 1}] \to \mathbb{K}$ is a specialisation 
     and let $\mathbb{K} \mathcal{H}^{\bf q}_{n}$ be the specialised algebra, where $\mathbb{K}$ is any field.  Note that for all $\blambda\in \Pi^l_n$, we have $s_{\blambda} ({\bf q})\in \mathbb{Z}[q^{\pm 1},Q_0^{\pm 1},\,\ldots,\,Q_{l-1}^{\pm 1}] $.
     Then  by \cite[Theorem 7.2.6]{GePf},
      $\mathbb{K}\mathcal{H}^{\bf q}_n$ is (split) semisimple if and only if, 
        for all $\blambda  \in \Pi^l_n$, we have $\theta (s_{\blambda} ({\bf q}))\neq 0$.
 From this,         we can deduce the following:
 
  \begin{Th}[Ariki]\label{ssimple}
Assume that $\mathbb{K}$ is a field. The algebra  $\mathbb{K} \mathcal{H}^{\bf q}_{n}$ is (split) semisimple if and only if $\theta ( P({\bf q}))\neq 0$,  where 
$$P({\bf q})=\prod_{1\leq i \leq n}(1+q+\cdots+q^{i-1}) \prod_{0 \leq s <t \leq l-1}\prod_{-n<k<n}(q^kQ_s-Q_t)$$

\end{Th}
\begin{proof}
Assume first that $\theta (P({\bf q}))=0 $. We distinguish three cases:
\begin{enumerate}[(a)]
\item If there exists $2\leq i \leq n$ such that $\theta (1+q+\ldots+q^{i-1})=0$, then we have $\theta ( [h^{\lambda^0,\lambda^0}_{1,n-i+1}]_q)=0$ for  $\blambda=((n),\emptyset,\ldots,\emptyset) \in \Pi^l_n$. Thus, for this $l$-partition, we have $\theta (s_{\blambda} ({\bf q}))=0$, which implies that  $\mathbb{K} \mathcal{H}^{\bf q}_{n}$ is not semisimple.
\item If there exist $0\leq s < t \leq l-1$ and $0 \leq k <n $
  such that $\theta(q^kQ_s-Q_t )=0$, then we have $\theta ( q^{h_{1,n-k}^{\lambda^s,\lambda^t}} Q_sQ_t^{-1}-1)=0$ for  $\blambda\in \Pi^l_n$  such that $\lambda^s=(n)$, $\lambda^t=\emptyset$.
  We have $\theta (s_{\blambda} ({\bf q}))=0$ and  $\mathbb{K} \mathcal{H}^{\bf q}_{n}$ is not semisimple.
 \item If there exist $0\leq s < t \leq l-1$ and $-n<k<0$  such that $\theta ( q^kQ_s-Q_t )=0 $, then we have $\theta ( q^{h_{1,n+k}^{\lambda^t,\lambda^s}}Q_tQ_s^{-1}-1)=0$ for  $\blambda\in \Pi^l_n$  such that $\lambda^s=\emptyset$, $\lambda^t=(n)$. Again, we have $\theta (s_{\blambda} ({\bf q}))=0$ and  $\mathbb{K} \mathcal{H}^{\bf q}_{n}$ is not semisimple.
 \end{enumerate}
Conversely, if  $\mathbb{K} \mathcal{H}^{\bf q}_{n}$ is not semisimple, then  there exists $\blambda  \in \Pi^l_n$ such that  $\theta (s_{\blambda} ({\bf q}) )= 0$. As 
for all $0\leq s,t\leq l-1$ and $(i,j)\in [\lambda^s]$, we have $-n<h_{i,j}^{\lambda^s,\lambda^t}<n$, we conclude that $\theta ( P({\bf q}))=0$.
\end{proof}
\end{abs}

\begin{abs}\label{set} We now consider a remarkable specialisation of the generic Ariki-Koike algebra. 
Let $u$ be an indeterminate. Let $r\in \mathbb{Z}_{>0}$ and let $r_0, \ldots, r_{l-1}$ be any integers. Set ${\bf r}:=(r_0,\ldots,r_{l-1})$ and $\eta_l:=\text{exp}(2\sqrt{-1}\pi/l)$. 
 For all $i=0,\ldots,l-1$, we set $m_i:=r_i/r$ and we define ${\bf m}:=(m_0,\ldots,m_{l-1})\in \mathbb{Q}^l$.  
  Assume that  $R=\mathbb{Z}[q^{\pm 1}, Q_0^{\pm 1},\ldots,Q_{l-1}^{\pm 1}]$ and consider the morphism
  $$\theta :R \to \mathbb{Z}[\eta_l][u^{\pm 1}]$$
  such that $\theta(q)=u^r$ and $\theta (Q_j)=\eta_l^j u^{r_j}$ for $j=0,1,\ldots,l-1$.
  We will denote by $\mathcal{H}_n^{{\bf m},r}$
  the specialisation of the Ariki-Koike algebra $\mathcal{H}^{\bf q}_n$
  via $\theta$. The algebra $\mathcal{H}_n^{{\bf m},r}$
   is called  a {\em cyclotomic Ariki-Koike algebra}. It is defined over $\mathbb{Z}[\eta_l][u^{\pm 1}]$
 and has   a presentation as follows:
\begin{itemize}
\item generators: $T_0$, $T_1$,\ldots, $T_{n-1}$,
\item relations: $$\begin{array}{rl}
(T_0-u^{r_0})(T_0-\eta_l u^{r_1})\cdots (T_0-\eta_l^{l-1} u^{r_{l-1}})=0& \\ &\\
(T_j-u^r)(T_j+1)=0& \textrm{for } j=1,...,n-1
\end{array}$$
and the braid relations symbolised by the  diagram
\begin{center}
\begin{picture}(250,17)
\put( 40, 03){\circle*{5}}
\put( 40, 03){\line(1,0){30}}
\put( 54, 05){$\scriptstyle{4}$}
\put( 70, 03){\circle*{5}}
\put( 70, 03){\line(1,0){30}}
\put(100, 03){\circle*{5}}
\put(100, 03){\line(1,0){20}}
\put(135, 00){$\cdot$}
\put(145, 00){$\cdot$}
\put(155, 00){$\cdot$}
\put(170, 03){\line(1,0){20}}
\put(190, 03){\circle*{5}}
\put( 37, 10){$T_0$}
\put( 67, 10){$T_1$}
\put( 97, 10){$T_2$}
\put(187, 10){$T_{n-1}$}
\end{picture}.
\end{center}
\end{itemize}

We set $K:=\mathbb{Q}(\eta_l)$. The algebra $K(u)\mathcal{H}_n^{{\bf m},r}$, which is obtained by extension of scalars to    $K(u)$,   is a  split semisimple algebra. As a consequence, 
one can apply Tits's Deformation Theorem (see, for example, \cite[\S 7.4]{GePf}), and obtain that
  $$\text{Irr}(K(u)\mathcal{H}_n^{{\bf m},r})=\{ V^{\blambda}\ |\ \blambda\in \Pi^l_n\}.$$ 

Using the Schur elements,
one can attach to every simple $K(u)\mathcal{H}_n^{{\bf m},r}$-module
  $V^{\blambda}$  a rational number ${\bf a}^{({\bf m},r)} ({\blambda})$, by setting
   ${\bf a}^{({\bf m},r)} ({\blambda})$ to be the
the negative of the valuation
  of the {Schur element} of $V^{\blambda}$ in $u$, that is, the
  negative of the valuation of $\theta (s_{\blambda} ({\bf q}))$.
We call this number
 the ${\bf a}$-{value of}   ${\blambda}$ .  
 By \cite[\S 5.5]{book}, this value may be easily computed combinatorially:
Let $\blambda\in \Pi^l_n$ and let $s\in \mathbb{Z}_{>0}$ be such that 
\begin{center}
$s\geq \text{min} \{i\in \mathbb{Z} \ |\ \forall j\in \{0,1,\ldots,l-1\},\ \lambda_i^j=0\}$.
\end{center} 
  Let $\mathfrak{B}$ be the shifted ${\bf m}$-symbol of $\blambda$ of size $s\in  \mathbb{Z}_{>0}$. 
      This is the $l$-tuple $(\mathfrak{B}^0,\ldots,\mathfrak{B}^{l-1})$
    where, for all $j=0,\ldots,l-1$ and  for all $i=1,\ldots, s+[m_j]$ (where $[m_j]$ denotes
     the integer part of $m_j$), we have
\begin{center}
$\mathfrak{B}^j_i=\lambda^j_i-i+s+m_{j}$ \,\,and\,\, $\mathfrak{B}^j=(\mathfrak{B}_{s+[m_j]}^j,\ldots,\mathfrak{B}_{1}^j)$ .
\end{center}
   Write 
   $$\kappa_1\geq \kappa_2 \geq \cdots \geq \kappa_h$$
   for the elements of $\mathfrak{B}$ written in decreasing order (allowing repetitions), where $h=ls+\sum_{0\leq j \leq l-1}[m_j]$. Let $\kappa_{\bf m}(\blambda)=(\kappa_1,\ldots,\kappa_h)\in \mathbb{Q}_{\geq 0}^h$ 
     and define 
   $$n_{\bf m}(\blambda):=\sum_{1\leq i\leq h}  (i-1)\kappa_i.$$
   Then, by  \cite[Proposition 5.5.11]{book}, the ${\bf a}$-value of $\blambda$ is :
   $${\bf a}^{({\bf m},r)} (\blambda)=r(n_{\bf m}(\blambda)-n_{\bf m}(\emptyset)).$$
   
   Generalising the dominance order for partitions, we will write 
   $\kappa_{\bf m}(\blambda)\rhd \kappa_{\bf m}(\bmu)$
   if $\kappa_{\bf m}(\blambda)\neq \kappa_{\bf m}(\bmu)$
   and $\sum_{1\leq i \leq t}\kappa_i(\blambda)\geq\sum_{1\leq i \leq t } \kappa_i(\bmu)$
 for all $t \geq 1$.
The following result  \cite[Proposition 5.5.16]{book} will be useful in then next sections:       \end{abs}
    \begin{Prop}\label{computatin}
    Assume that $\blambda$ and $\bmu$ are two $l$-partitions with the same rank  such that $\kappa_{\bf m}(\blambda)\rhd \kappa_{\bf m}(\bmu)$.
     Then ${\bf a}^{({\bf m},r)} (\bmu)>{\bf a}^{({\bf m},r)} (\blambda)$. 
    \end{Prop}

Now, Formula (\ref{claim}) allows us to give an alternative description of the ${\bf a}$-value of $\blambda$:
  \begin{Prop}
    Let  $\blambda\in \Pi^{l}_n$. The ${\bf a}$-value of $\blambda$ is  
     $${\bf a}^{({\bf m},r)} (\blambda)=r \left(n(\overline{\blambda})-\sum_{0\leq s\leq l-1} \sum_{(i,j)\in [\lambda^s]} \sum_{0\leq t\leq l-1, t\neq s} \operatorname{min }(h^{\lambda^s,\lambda^t}_{i,j}+m_s-m_t,0)\right).$$
    \end{Prop}

    \begin{abs}
We  now consider another type of specialisation. Let $v_0, \ldots, v_{l-1}$ be any integers. Let $k$ be a subfield of $\mathbb{C}$ and let $\eta$ be a primitive root of unity of order $e>1$. 
Assume that  $R=\mathbb{Z}[q^{\pm 1}, Q_0^{\pm 1},\ldots,Q_{l-1}^{\pm 1}]$ and consider the morphism
 $$\theta : R  \to k(\eta) $$ such that
  $\theta (q)=\eta$ and $\theta(Q_j)=\eta^{v_j}$ for $j=0,1,\ldots, l-1$. By Theorem \ref{ssimple}, the specialised  algebra  $k(\eta)  \mathcal{H}_n^{\bf q}$  is not generally semisimple, and 
 a result by Dipper and Mathas which will be specified  later (see \S \ref{morita}) implies  that the study of this algebra is enough for studying the non-semisimple 
  representation theory of Ariki-Koike algebras in characteristic $0$. Let 
$$D=([V^{\blambda}:M])_{\blambda\in \Pi^l_n,\,M\in \text{Irr} ( k(\eta)  \mathcal{H}_n^{\bf q})}$$
\\
 be the associated decomposition matrix (see \cite[\S 7.4]{GePf}), which relates the irreducible representations of the split semisimple Ariki-Koike algebra $\mathcal{H}_n^{{\bf q}}$
 and the specialised Ariki-Koike algebra $k(\eta)  \mathcal{H}_n^{{\bf q}}$. We are interested in the classification of the blocks of defect $0$.
  That is, we want to classify the $l$-partitions $\blambda \in \Pi^l_n$ which are alone in their blocks in the decomposition matrix.  These correspond to the modules $V^{\blambda}$ which remain projective and irreducible after the specialisation $\theta$. By  \cite[Lemme 2.6]{MaRou}, these elements are characterized by the property that 
    $\theta (s_{\blambda} ({\bf q}))\neq 0$.
     In our setting, using Formula (\ref{claim}), we obtain the following:
    \begin{Prop}
   Under the above hypotheses,  $\blambda\in \Pi^l_n$ is in a block of defect $0$ if and only if, for all  
    $0\leq s,t\leq l-1$ and $(i,j)\in [\lambda^s]$,      $e$ does not divide $h^{\lambda^s,\lambda^t}_{i,j}+v_s-v_t$. 
    
    \end{Prop}

   \begin{Rem} As pointed out by M. Fayers and A. Mathas, the above proposition  should also be obtained using \cite{Fayers}. 
   
   \end{Rem}

\end{abs}

\end{section}
\begin{section}{Canonical basic sets for Ariki-Koike algebras}
In this part, we generalise some known results on basic sets for Ariki-Koike algebras, using a fundamental result by Dipper and Mathas. This will help us determine 
 the canonical basic sets for cyclotomic Ariki-Koike algebras in full generality.

    \begin{abs}
    We consider the cyclotomic Ariki-Koike algebra $\mathcal{H}_n^{{\bf m},r}$ defined in \S\ref{set}, replacing from now on the indeterminate $u$ by the indeterminate $q$ (following the usual notation). Let  $\theta : \mathbb{Z}[\eta_l][q^{\pm 1}] \to K(\eta) $ be a specialisation such that $\theta (q)=\eta\in \mathbb{C}^*$.
   We obtain  a specialised Ariki-Koike algebra   $K(\eta)  \mathcal{H}_n^{{\bf m},r}$. The relations
    between the generators are the usual braid relations together with the following ones:
$$\begin{array}{rl}
(T_0-\eta^{r_0})(T_0-\eta_l \eta^{r_1})\cdots (T_0-\eta_l^{l-1} \eta^{r_{l-1}})=0& \\ &\\
(T_j-\eta^r)(T_j+1)=0& \textrm{for } j=1,...,n-1.
\end{array}$$
Let 
$$D=([V^{\blambda}:M])_{\blambda\in \Pi^l_n,\,M\in \text{Irr} ( K(\eta)  \mathcal{H}_n^{{\bf m},r})}$$
\\
 be the associated decomposition matrix (see \cite[\S 7.4]{GePf}). The matrix $D$ relates the irreducible representations of the split semisimple Ariki-Koike algebra $K(q)  \mathcal{H}_n^{{\bf m},r}$
 and the specialised Ariki-Koike algebra $K(\eta)  \mathcal{H}_n^{{\bf m},r}$.
  The goal of this section is to study the form of this matrix in full generality.
  
 First assume that $\eta$ is not a root of unity. 
Then, for all $0\leq i\neq j \leq l-1$, we have 
$$\eta_l^{i-j} \eta^{r_i-r_j}\neq \eta^{rd}$$
for all $d\in \mathbb{Z}_{>0}$. By the criterion of semisimplicty due to Ariki (Theorem \ref{ssimple}), this implies that the algebra $K(\eta)  \mathcal{H}_n^{{\bf m},r}$  is split semisimple,
 and thus $D$ is the identity matrix. Hence, from now, one may assume that $\eta$ is a primitive root of unity of order $e>1$. Then there exists $k\in \mathbb{Z}_{>0}$ such that 
 $\mathrm{gcd}(k,e)=1$ and $\eta=\text{exp}(2\sqrt{-1}\pi k/e)$.

\end{abs}

\begin{abs}\label{morita}
We will now use a reduction theorem by  Dipper and Mathas which will help us understand the form of $D$.  
 Set ${\bf I}=\{0,1,\ldots,l-1\}$. 
   There is a partition
  $${\bf I}={\bf I}_1\sqcup {\bf I}_2 \sqcup \ldots \sqcup {\bf I}_p$$
  such that
  $$ \prod_{1 \leq \alpha < \beta\leq p} \prod_{(i,j) \in {\bf I}_\alpha \times {\bf I}_\beta}
  \prod_{-n<d<n} (\eta^{rd}-\eta_l^{i-j}\eta^{r_i-r_j}) \neq 0.$$
For all $i=1,\ldots,p$, we set  $l_i:= | {\bf I}_i|$ and we consider ${\bf I}_i$ as an ordered set
$${\bf I}_{i}=({i_1}, {i_{2}},\ldots,{i_{l_i}})  \text{ with } i_1<i_2<\cdots<{i_{l_i}}.$$
We define
   $$\begin{array}{lccc}
    \pi_i : & \mathbb{Q}^l & \to & \mathbb{Q}^{l_i} \\
     & (x_0,x_1,\ldots,x_{l-1}) & \mapsto & (x_{i_1}, x_{i_{2}},\ldots,x_{i_{l_i}})
     \end{array}$$
         
For $n_i\in \mathbb{Z}_{\geq 0}$,   we have an Ariki-Koike algebra  of type $G(l_i,1,n_i)$ which we denote by  ${\bf H}_{n_i}^{{\bf m}^i,r}$ with
   ${\bf m}^i:=\pi_i ({\bf m})=(m_{i_1}, m_{i_{2}},\ldots,m_{i_{l_i}})$. The relations
    between the generators are the usual braid relations together with the following ones:
$$\begin{array}{rl}
(T_0-\eta_l^{i_1} q^{r_{i_1}})(T_0-\eta^{i_2}_l q^{r_{i_2}})\cdots (T_0-\eta_l^{i_{l_i}} q^{r_{i_{l_i}}})=0& \\ &\\
(T_j-q^r)(T_j+1)=0& \textrm{for } j=1,...,n_i-1.
\end{array}$$
   Note however that   ${\bf H}_{n_i}^{{\bf m}^i,r}$ is not a cyclotomic Ariki-Koike algebra in general, because $l_i\neq l$.
       The specialisation $\theta_i : \mathbb{Z}[\eta_{l}][q^{\pm 1}] \to K(\eta) $
    such that $\theta (q)=\eta$  defines a specialised algebra  $K(\eta)  {\bf H}_{n_i}^{{\bf m}^i,r}$, and we have  an 
    associated decomposition matrix 
    $$D^i_{n_i}=([V^{\blambda}:M])_{\blambda\in \Pi^{l_i}_{n_i},\,M\in \text{Irr} ( K(\eta)  {\bf H}_{n_i}^{{\bf m}^i,r})}.$$
   
   In \cite{DiMa}, Dipper and Mathas have shown that $K(\eta)  \mathcal{H}_n^{{\bf m},r}$  is Morita equivalent to the algebra
   $$\bigoplus_{
   \begin{array}{c}
   n_1,\ldots,n_p \geq 0\\
   n_1+\cdots+n_p=n
   \end{array}} K(\eta)  {\bf H}_{n_1}^{{\bf m}^1,r}\otimes_{K(\eta)} K(\eta)  {\bf H}_{n_2}^{{\bf m}^2,r}\otimes_{K(\eta)}    \cdots \otimes_{K(\eta)}  K(\eta)  {\bf H}_{n_p}^{{\bf m}^p,r}.$$
   Thus, for a suitable ordering of the rows and columns, $D$ has the form of a block diagonal matrix where 
 each block is given by $D^1_{n_1}\otimes \cdots \otimes D^p_{n_p}$ with $n_1+\ldots +n_p=n$. More precisely, we have the following result (\cite[Proposition 4.11]{DiMa}):
\begin{Th}[Dipper-Mathas]\label{Dimathas}
Let $\blambda \in \Pi^l_n$ and $M\in \operatorname{Irr} ( K(\eta)  \mathcal{H}_{n}^{{\bf m},r})$. There exist integers $n_1,\ldots,n_p \geq 0$ with $n_1+\cdots+n_p=n$, and
$M_i\in \operatorname{Irr} ( K(\eta)  \mathcal{H}_{  n_i   }^{{\bf m}^i,r})$
such that 
 $$[V^{\blambda}:M]=
 \left\{\begin{array}{ll}
 \prod_{1\leq i \leq p}  [V^{\pi_{i}(\blambda)}:M_i],& \text{ if }\, \pi_{i}(\blambda) \in \Pi_{n_i}^{l_i}\,\,
 \forall i \in [1,p]
 \\ &\\
 0,& \text{ otherwise.}
  \end{array}\right.$$
 \end{Th}

\end{abs}

\begin{abs}\label{s}
We now fix   $i \in \{1,\ldots, p\}$. 
By definition of ${\bf I}_{i}=({i_1}, {i_{2}},\ldots,{i_{l_i}})$, 
one may assume that,  for all $j=1,\ldots,{l_i}$,  there exist $s_j \in \mathbb{Z}$ such that 
$$\eta_l^{i_j-i_1}\eta^{r_{i_j}-r_{i_1}}
=\eta^{rs_j}$$
(with $s_1=0$). 
We have
$$r_{i_j}-r_{i_1}=rs_j -e (i_j-i_1)/(kl),$$
whence we deduce the following relation:
\begin{equation}\label{rela}
m_{i_j}-m_{i_1}=s_j -e (i_j-i_1)/(klr).
\end{equation}
Set ${\bf s}^i:=(s_1,\ldots,s_{l_i})$. 
\end{abs}

\begin{abs} Keeping the above notation, let us consider the Ariki-Koike algebra $K(\eta)  \mathcal{H}_{n_i}^{{\bf m}^i,r}$  of type $G(l_i,1,n_i)$  (with $n_i\leq n$) with relations
$$\begin{array}{rl}
(T_0-\eta^{i_1}_l \eta^{r_{i_1}})(T_0-\eta^{i_2}_l \eta^{r_{i_2}})\cdots (T_0-\eta_l^{i_{l_i}} \eta^{r_{i_{l_i}}})=0& \\&\\
(T_j-\eta^r)(T_j+1)=0& \textrm{for } j=1,...,n_i-1.
\end{array}$$
This is then isomorphic to the Ariki-Koike algebra  with relations
$$\begin{array}{rl}
(T_0-\eta^{r s_{1}})(T_0-\eta^{r s_{2}})\cdots (T_0- \eta^{r {s_{l_i}}})=0& \\& \\
(T_j-\eta^r)(T_j+1)=0& \textrm{for } j=1,...,n_i-1.
\end{array}$$
The following is a direct consequence of  \cite[Theorem 6.7.2]{book}.
\begin{Prop}\label{conse} Under the above hypothesis, there exists a set $\Phi^{l_i}_{n_i}({\bf s}^i)  \subset \Pi^{l_i}_{n_i}$ with 
$$| \Phi^{l_i}_{n_i} ({\bf s}^i)  |=|\operatorname{Irr} (  K(\eta)  \mathcal{H}_{n_i}^{{\bf m}^i,r})|$$
such that the following property is satisfied: For  any  $M\in \operatorname{Irr} (K(\eta)  \mathcal{H}_{n_i}^{{\bf m}^i,r})$, 
 there exists a unique $l_i$-partition $\blambda_M \in\Phi^{l_i}_{n_i} ({\bf s}^i) $ such that
 \begin{itemize}
\item $[V^{\blambda_M}:M]=1$ \text{ and }  \smallbreak
\item  $[V^{\blambda}:M]\neq 0$ for $\blambda \in\Pi^{l_i}_{n_i}$ only if  $\kappa_{{\bf m}^i}(\blambda_M)\triangleright  \kappa_{{\bf m}^i}(\blambda) $ or $\blambda=\blambda_M$.
\end{itemize}
\end{Prop}
\begin{proof}
For $j=1,2,\ldots,l_i$, set $\underline{m}_{i_j}:=s_{j}-ei_j/(klr)$ and $\underline{\bf m}^i:=(\underline{m}_{i_1},\ldots,\underline{m}_{i_{l_i}})$. By \cite[Theorem 6.7.2]{book}, there exists 
 a set $\Phi^{l_i}_{n_i}  \subset \Pi^{l_i}_{n_i}$  satisfying the property of the proposition except that 
   $\kappa_{{\bf m}^i}$ is replaced by  $\kappa_{\underline{\bf m}^i}$. By Equality (\ref{rela}), we have
   $$\underline{m}_{i_j}=m_{i_j}-m_{i_1}+\underline{m}_{i_1}.$$ 
   It easily follows that 
   $ \kappa_{{\bf m}^i}(\blambda_M)\triangleright  \kappa_{{\bf m}^i}(\blambda) $ if and only if  $\kappa_{\underline{\bf m}^i}(\blambda_M)\triangleright  \kappa_{\underline{\bf m}^i}(\blambda)$, which yields the desired result. 
\end{proof}
\end{abs}
\begin{abs}
We now need an easy combinatorial lemma. In this section, all
 multisets of rational numbers are ordered so that their elements form decreasing sequences. 
Moreover, if $X$ and $Y$ are multisets, we will write $X \sqcup Y$ for the multiset consisting of all the elements of $X$ and $Y$ together, so that $|X \sqcup Y|=|X|+|Y|$. 

\begin{Lem}\label{Combi}
Let $\mu$ and $\nu$ be two multisets of positive rational numbers. Assume that there exist multisets $\mu^1,\mu^2,\ldots,\mu^h$ and $\nu^1,\nu^2,\ldots,\nu^h$ such that 
$$\mu=\bigsqcup_{i=1}^h \mu^i,\,\,\,\nu=\bigsqcup_{i=1}^h \nu^i\,\,\,\text{ and }\,\,\,
\mu^i \trianglerighteq \nu^i\text{ for all } i=1,\ldots,h.$$
Then $\mu \trianglerighteq \nu$ (with the equality holding only when  $\mu^i=\nu^i$ for all $i=1,\ldots,h$).
\end{Lem}
\begin{proof} If $h=1$, there is nothing to prove. Suppose that $h=2$, and let $t \geq 1$. We have $\sum_{1\leq j \leq t} \mu_j = \sum_{1\leq j \leq t} \mu_j^1
+\sum_{1\leq j \leq t_2} \mu_j^2$ for some $t_1, t_2 \geq 1$ such that $t_1+t_2=t$,
and $\sum_{j=1}^t \nu_j = \sum_{1\leq j \leq t_1 '} \nu_j^1
+\sum_{1\leq j \leq t_2'} \nu_j^2$ for some $t_1', t_2' \geq 1$ such that $t_1'+t_2'=t$.
Suppose that $t_1 \geq t_1'$. Then $t_2 \leq t_2'$, and we have
$$ \sum_{1\leq j \leq t}\mu_j=\sum_{1\leq j \leq t_1 '} \mu_j^1+\sum_{t_1'+1\leq j \leq t_1} \mu_j^1+
\sum_{1\leq j \leq t_2} \mu_j^2 $$
Now, 
$$\begin{array}{rcl}
\displaystyle \sum_{1\leq j \leq t_1 '} \mu_j^1+\sum_{t_1'+1\leq j \leq t_1} \mu_j^1+
\displaystyle\sum_{1\leq j \leq t_2} \mu_j^2 &\geq& \displaystyle\sum_{1\leq j t_1'} \mu_j^1+\sum_{t_2+1\leq j \leq  t_2'} \mu_j^2+
\displaystyle\sum_{1\leq j \leq t_2} \mu_j^2 \\ 
&\geq& \displaystyle \sum_{1\leq j \leq t_1'} \nu_j^1
+\sum_{1 \leq j \leq t_2'} \nu_j^2. \end{array} $$
and we can conclude because 
$$ \displaystyle \sum_{1\leq j \leq t_1'} \nu_j^1
+\sum_{1 \leq j \leq t_2'} \nu_j^2 =\sum_{1\leq j \leq t} \nu_j$$
Induction yields the result for $h>2$.
\end{proof}
\end{abs}
\begin{abs}
We are now in position to prove the main result of this section:
\begin{Th}\label{main}
In the setting of \S \ref{set}, 
the algebra $\mathcal{H}_n^{{\bf m},r}$ admits a canonical basic set $\mathcal{B}_{\theta}$ with respect to any specialisation  $\theta : \mathbb{Z}[\eta_l][q^{\pm 1}] \to K(\eta) $
such that $\theta (q)=\eta\in \mathbb{C}^*$, i.e., 
 there exists a set $\mathcal{B}_{\theta}  \subset \Pi^{l}_{n}$ with
$$| \mathcal{B}_{\theta}   |=|\operatorname{Irr} (  K(\eta)  \mathcal{H}_{n}^{{\bf m},r})|$$
such that the following property is satisfied: For  any  $M\in \operatorname{Irr} ( K(\eta)  \mathcal{H}_{n}^{{\bf m},r})$, 
 there exists a unique $l$-partition $\blambda_M \in\mathcal{B}_{\theta}  $ such that  
  \begin{itemize}
\item $[V^{\blambda_M}:M]=1$ \text{ and }  \smallbreak
\item  $[V^{\blambda}:M]\neq 0$ for $\blambda \in\Pi^{l}_{n}$ only if  ${\bf a}^{({\bf m},r)}(\blambda) > {\bf a}^{({\bf m},r)}(\blambda_M)$ or $\blambda=\blambda_M$.
\end{itemize}
In addition, we have that $\blambda\in \mathcal{B}_{\theta}$ if and only if there exist integers $n_1,\ldots,n_p \geq 0$ with $n_1+\cdots+n_p=n$ such that, for all $i=1,\ldots,p$,
$\pi_i (\blambda)\in \Phi_{n_i}^{l_i} ({\bf s}^i)$.
\end{Th}
\begin{proof}
Let 
$${B}_{\theta}= \{\blambda\in \Pi^l_n\,|\, \exists \,n_1,\ldots,n_p  \geq 0,
n_1+\cdots+n_p=n:
 \forall i \in [1,p],\ \pi_i (\blambda)\in \Phi_{n_i}^{l_i} ({\bf s}^i)\}.$$  First note that, by 
 \S \ref{morita}, we have 
 $$| \mathcal{B}_{\theta} |=|\operatorname{Irr} (  K(\eta)  \mathcal{H}_{n}^{{\bf m},r})|.$$

Let  $\blambda \in \Pi^l_n$ and $M\in \operatorname{Irr} ( K(\eta)  \mathcal{H}_{n}^{{\bf m},r})$. By Proposition \ref{Dimathas}, 
 there exist integers $n_1,\ldots,n_p \geq 0$ with $n_1+\cdots+n_p=n$, and
$M_i\in \operatorname{Irr} ( K(\eta)  \mathcal{H}_{  n_i   }^{{\bf m}^i,r})$
such that 
\begin{equation}\label{eq2}
 [V^{\blambda}:M]=
 \left\{\begin{array}{ll}
 \prod_{1\leq i \leq  p} [V^{\pi_{i}(\blambda)}:M_i],& \text{ if }\, \pi_{i}(\blambda) \in \Pi_{n_i}^{l_i}\,\,
 \forall i \in [1,p]
 \\ &\\
 0,& \text{ otherwise.}
  \end{array}\right.
\end{equation} 
  
  We consider the $l$-partition $\blambda_M\in \Pi^l_n$ such that $\pi_i (\blambda_M)=\blambda_{M_i}$ for all $i=1,\ldots,p$, where $\blambda_{M_i}\in \Phi^{l_i}_{n_i} ({\bf s}^i)$
  is defined in Proposition \ref{conse}. Note that we have $\blambda_M\in \mathcal{B}_{\theta}$ and
  $ [V^{\blambda_M}:M]=1$. 
  
  Now let $\blambda \in\Pi^l_{n}$, $\blambda \neq \blambda_M$.
 Following Equation  (\ref{eq2}) and Proposition \ref{conse}, if $[V^{\blambda}:M] \neq 0$, then, for all $i=1,\ldots,p$, 
 \begin{center}
 either \,\,$\pi_i(\blambda)=\blambda_{M_i}=\pi_i(\blambda_M)$\,\,\,\,\,or\,\,   $ \kappa_{{\bf m}^i}(\pi_i (\blambda_M))\rhd  \kappa_{{\bf m}^i}(\pi_i(\blambda))$. 
  \end{center}
By the definition of $\kappa_{\bf m}$, we have $ \kappa_{{\bf m}} (\blambda)= \bigsqcup_{i=1}^p
 \kappa_{{\bf m}^i}(\pi_i (\blambda))$. By Lemma 
  \ref{Combi}, since $\blambda \neq \blambda_M$, we must have $  \kappa_{{\bf m}} (\blambda_M)\rhd  \kappa_{{\bf m}}(\blambda)$.
  The result follows now from Proposition \ref{computatin}.
\end{proof}
\begin{Rem}
If $r=1$, then the elements of $\mathcal{B}_{\theta}$ are the $e$-\emph{Uglov $l$-partitions of $n$}.
(cf.~\cite[Definition 3.2]{jaca}). For $r>1$, we will refer to the elements of $\mathcal{B}_{\theta}$ as \emph{generalised $e$-Uglov $l$-partitions of $n$}.
\end{Rem}
\end{abs}

\begin{abs}
Let us give an example, where we calculate the canonical basic set 
when $n=2$, $r=6$ and ${\bf m} = (1/2, -1/6, -1/3)$.
Consider the cyclotomic  Ariki-Koike algebra $\mathcal{H}_{2}^{{\bf m},r}$ of type $G(3,1,2)$, with generators $T_0$, $T_1$ and relations
$$T_0T_1T_0T_1=T_1T_0T_1T_0,\,\,(T_0-q^3)(T_0-\eta_3q^{-1})(T_0-\eta_3^2q^{-2})=0,\,\,\,
(T_1-q^6)(T_1+1)=0.$$
Let $\theta: \mathbb{Z}[\eta_3][q^{\pm 1}] \rightarrow \mathbb{Q}(\eta_{12})$ be a specialisation such that $\theta(q)=\eta_{12}$ (we have $e=12$ and $k=1$). Then the specialised Ariki-Koike algebra 
$\mathbb{Q}(\eta_{12})\mathcal{H}_{2}^{{\bf m},r}$ is generated by $T_0$ and $T_1$ with relations
$$T_0T_1T_0T_1=T_1T_0T_1T_0,\,\,(T_0-i)^2(T_0+1)=0,\,\,\,
(T_1+1)^2=0.$$
By \cite[Theorem 1.1]{DiMa}, the specialised Ariki-Koike algebra 
$\mathbb{Q}(\eta_{12})  \mathcal{H}_{2}^{{\bf m},r}$ is Morita equivalent to the algebra
$$\mathbb{Q}(\eta_{12})  {\bf H}_{2}^{{\bf m}^1,r}  \oplus \left( \mathbb{Q}(\eta_{12})  {\bf H}_{1}^{{\bf m}^1,r}\otimes
\mathbb{Q}(\eta_{12}) {\bf H}_{1}^{{\bf m}^2,r} \right) \oplus \mathbb{Q}(\eta_{12})  {\bf H}_{2}^{{\bf m}^2,r},$$
where ${\bf m}^1=(1/2,-1/6)$ and ${\bf m}^2=(-1/3)$. We now have that
\begin{itemize}
\item the algebra $\mathbb{Q}(\eta_{12})  {\bf H}_{2}^{{\bf m}^1,r}$ 
is isomorphic to the cyclotomic Ariki-Koike algebra of type $G(2,1,2) \cong B_2$
with generators
$T_0$ and $T_1$, and relations
$$T_0T_1T_0T_1=T_1T_0T_1T_0,\,\,(T_0-i)^2=0,\,\,\,
(T_1+1)^2=0,$$
\item the algebra $\mathbb{Q}(\eta_{12})  {\bf H}_{1}^{{\bf m}^1,r}$
is isomorphic to the cyclotomic Ariki-Koike algebra of type $G(2,1,1) \cong \mathbb{Z}/2\mathbb{Z}$ with quadratic relation $(T_0-i)^2=0$, 
\item the algebra $\mathbb{Q}(\eta_{12})  {\bf H}_{1}^{{\bf m}^2,r}$ is isomorphic to the algebra of the trivial group, and 
\item the algebra $\mathbb{Q}(\eta_{12})  {\bf H}_{2}^{{\bf m}^2,r}$ is isomorphic to the cyclotomic Ariki-Koike algebra of type $G(1,1,2) \cong \mathfrak{S}_2$ with quadratic relation $(T_1+1)^2=0$.
\end{itemize} 
Keeping the notation of \S\ref{s} and Proposition  \ref{conse}, we obtain:
\begin{itemize}
\item $\Phi_{2}^2({\bf s}^1)= \{((2),\emptyset),\,((1),(1))\}$, \smallbreak
\item $\Phi_{1}^2({\bf s}^1)=\{((1),\emptyset)\}$,  \smallbreak
\item $\Phi_{1}^1({\bf s}^2)=\{(1)\}$,  \smallbreak
\item $\Phi_{2}^1({\bf s}^2)=\{(2)\}$.  \smallbreak
\end{itemize}
Therefore, the canonical basic set  with respect to ${\theta}$  for $\mathcal{H}_{2}^{{\bf m},r}$ is
$$\mathcal{B}_\theta=\{ ((2),\emptyset,\emptyset), ((1),(1),\emptyset), ((1),\emptyset,(1)),
(\emptyset,\emptyset,(2))\}.$$

\end{abs}

\end{section}
\begin{section}{Canonical basic sets for cyclotomic Hecke algebras of type $G(l,p,n)$}
The purpose of this last part is to deduce from the last section the existence of the explicit parametrisation of the basic sets for 
 Cyclotomic Hecke algebra of type $G(l,p,n)$.
\begin{abs}
Let $l,p,n$ be three positive integers with $n>2$ (we can also take $n=2$, but then we must assume that $p$ is odd). Set $d:=l/p$. Let $r \in \mathbb{Z}_{>0}$ and let $r_0,r_1,\ldots,r_{d-1}$ be any integers.
For all $i=0,\ldots,d-1$, we set $m_i:=r_i/(pr)$ and we define
${\bf m}:=(m_0,\ldots,m_{d-1},m_0,\ldots,m_{d-1},\ldots, m_0,\ldots,m_{d-1}) \in \mathbb{Q}^l$
(where the $d$-tuple $(m_0,\ldots,m_{d-1})$ is repeated $p$ times). We consider the cyclotomic Hecke algebra $\mathcal{H}_{p,n}^{{\bf m},pr}$ of type $G(l,p,n)$  over $\mathbb{Z}[\eta_l][q^{\pm 1}]$ with presentation as follows:
\begin{itemize}
\item generators: $t_0$, $t_1$,\ldots, $t_{n}$,
\item relations: $$\begin{array}{rl}
(t_0-q^{pr_0})(t_0-\eta_d q^{pr_1})\cdots (t_0-\eta_d^{d-1} q^{pr_{d-1}})=0& \\ &\\
(t_j-q^{pr})(t_j+1)=0& \textrm{for } j=1,...,n
\end{array}$$
and the braid relations
\begin{itemize}
\item $t_1t_3t_1=t_3t_1t_3$, $t_jt_{j+1}t_j=t_{j+1}t_jt_{j+1}$ for $j=2,\ldots,n-1$,
\item $t_1t_2t_3t_1t_2t_3=t_3t_1t_2t_3t_1t_2$,
\item $t_1t_j=t_jt_1$ for $j=4,\ldots,n$,
\item $t_i t_j=t_j t_i$  for $2 \leq i <j \leq n$ with $j-i>1$,
\item $t_0t_j=t_jt_0$ for $j=3,\ldots,n$,
\item $t_0t_1t_2=t_1t_2t_0$,
\item $\underbrace{t_2t_0t_1t_2t_1t_2t_1\ldots}_{p+1 \textrm{ factors}}=
\underbrace{t_0t_1t_2t_1t_2t_1t_2\ldots}_{p+1 \textrm{ factors}}$\,.
\end{itemize}
\end{itemize}
\end{abs}

\begin{abs}\label{sigma}
Let us denote by $G$ the cyclic group of order $p$.
The algebra  $\mathcal{H}_{p,n}^{{\bf m},pr}$ can be viewed as a subalgebra of index $p$ of the cyclotomic Hecke algebra
$\mathcal{H}_{n}^{{\bf m},pr}$ of type $G(l,1,n)$: in fact, $\mathcal{H}_{n}^{{\bf m},pr}$  is a ``twisted symmetric algebra'' of $G$ over $\mathcal{H}_{p,n}^{{\bf m},pr}$ (see \cite[\S 5.5.1]{springer}).
The action of $G$ on $\mathrm{Irr}(K(q)\mathcal{H}_{n}^{{\bf m},pr})$ corresponds to the action generated by the cyclic permutation by $d$-packages on the $l$-partitions of $n$:
$$\begin{array}{rl}
\sigma: &\,\,\,\blambda=(\el^{0},\ldots,\el^{d-1},\el^{d},\ldots,\el^{2d-1},\ldots,\el^{pd-d},\ldots,\el^{pd-1})\\ &\\
\mapsto &{}^{\sigma}\blambda=(\el^{pd-d},\ldots,\el^{pd-1},\el^{0},\ldots,\el^{d-1},\ldots, \el^{pd-2d},\ldots,\el^{pd-d-1}).\smallbreak
\end{array}$$
By  \cite[Proposition 2.5]{CJ1}, we have ${\bf a}^{({\bf m},pr)}(\blambda)={\bf a}^{({\bf m},pr)}({}^\sigma\blambda)$.
\end{abs}
\begin{abs}
In this section, we will use extensively  some results known as ``Clifford theory''. For more details, the reader may refer to \cite[\S 2.3]{springer} and \cite{GenJa}. At the end, we will be able to deduce the 
 existence and the explicit parametrisation of a canonical basic set for $\mathcal{H}_{p,n}^{{\bf m},pr}$. The proof below is inspired from \cite[Proof of Theorem 3.1]{GenJa}.
From now on, we will write $\mathcal{H}$ for $\mathcal{H}_{n}^{{\bf m},pr}$ and $\bar{\mathcal{H}}$ for $\mathcal{H}_{p,n}^{{\bf m},pr}$. Let $\theta : \mathbb{Z}[\eta_l][q^{\pm 1}] \to K(\eta)$ be a specialisation such that $\theta (q)=\eta\in \mathbb{C}^*$. As before, one may assume that $\eta$ is a primitive root of unity of order $e>1$.

Let $E \in \mathrm{Irr}(K(q)\bar{\mathcal{H}})$. By Clifford theory, there exists $V^{\blambda} \in
 \mathrm{Irr}(K(q){\mathcal{H}})$ such that $E$ is a composition factor of $\mathrm{Res}_{\bar{\mathcal{H}}}^{\mathcal{H}}(V^{\blambda})$. We write $E^{\blambda}$ for $E$. Moreover, there is an action of $G$ on $\mathrm{Irr}(K(q)\bar{\mathcal{H}})$ such that, if we denote by $\bar{\Omega}_{\blambda}$ the orbit of $E^{\blambda}$ under the action of $G$, we have
 $$[\mathrm{Res}_{\bar{\mathcal{H}}}^{\mathcal{H}}(V^{\blambda})]=\sum_{E \in \bar{\Omega}_{\blambda}}[E].$$
 Let $V \in  \mathrm{Irr}(K(q){\mathcal{H}})$.
 The elements of $\bar{\Omega}_{\blambda}$ appear as composition factors in
 $\mathrm{Res}_{\bar{\mathcal{H}}}^{\mathcal{H}}(V)$ if and only if $V={}^gV^{\blambda}$ for some
 $g \in G$.
 In particular, if $\sigma$ is the map defined in \S \ref{sigma}, we have
 $$[\mathrm{Res}_{\bar{\mathcal{H}}}^{\mathcal{H}}({}^\sigma V^{\blambda})]=[\mathrm{Res}_{\bar{\mathcal{H}}}^{\mathcal{H}}(V^{\blambda})].$$ 
We deduce that
$$\mathrm{Irr}(K(q)\bar{\mathcal{H}})=\{ E\,|\,E \in  \bar{\Omega}_{\blambda}, \, \blambda \in \Pi^l_n\}.$$
 
 Now, if we denote by $\Omega_{\blambda}$ the orbit of $V^{\blambda}$ under the action of $G$, we have
 $$|\Omega_{\blambda}| |\bar{\Omega}_{\blambda}|=|G|=p.$$
 (see \cite[Lemma 2.2]{GenJa}). Thus, $|\bar{\Omega}_{\blambda}|=|G_{\blambda}|$, where  $G_{\blambda}: =\{g \in G\,|\,{}^g\blambda=\blambda\}$.
 Furthermore, applying the restriction functor $\mathrm{Res}_{\bar{\mathcal{H}}}^{\mathcal{H}}$ does not affect the ${\bf a}$-value of simple modules over $K(q)$
 (see \cite[Proposition 2.3.15]{springer}).
Hence, we obtain:
\begin{equation}\label{aMN}
{\bf a}^{({\bf m},pr)}(\blambda)={\bf a}^{({\bf m},pr)}({}^\sigma\blambda)={\bf a}^{({\bf m},pr)}(E),\text{ for all } E \in \bar{\Omega}_{\blambda}.
\end{equation}

Now, to each simple $K(\eta)\mathcal{H}$-module $M$, one can attach an ${\bf a}$-value as follows:
$$ {\bf a}^{({\bf m},pr)}(M)=\mathrm{min}\{{\bf a}^{({\bf m},pr)}(\blambda)\,\,|\,\,
[V^{\blambda}:M] \neq 0 \}.$$
Respectively, to each simple $K(\eta)\bar{\mathcal{H}}$-module $N$, one can attach an ${\bf a}$-value as follows:
$$ {\bf a}^{({\bf m},pr)}(N)=\mathrm{min}\{{\bf a}^{({\bf m},pr)}(E)\,\,|\,\,E \in \mathrm{Irr}(K(q)\bar{\mathcal{H}}),\,\,
[E:N] \neq 0 \}.$$

Let $N \in \mathrm{Irr}(K(\eta)\bar{\mathcal{H}})$. By Clifford theory, there exists $M \in
 \mathrm{Irr}(K(\eta){\mathcal{H}})$ such that $N$ is a composition factor of $\mathrm{Res}_{\bar{\mathcal{H}}}^{\mathcal{H}}(M)$. We write $N_M$ for $N$. There is an action of $G$ on $\mathrm{Irr}(K(\eta)\bar{\mathcal{H}})$ such that, if we denote by $\bar{\omega}_{M}$ the orbit of $N_M$ under the action of $G$, we have
 $$[\mathrm{Res}_{\bar{\mathcal{H}}}^{\mathcal{H}}(M)]=\sum_{N \in \bar{\omega}_{M}}[N].$$
By Theorem \ref{main}, the algebra $\mathcal{H}$ admits a canonical basic set $\mathcal{B}_{\theta}$ with respect to $\theta$. Thus, there exists $\blambda_M \in  \mathcal{B}_{\theta}$ such that the conditions of Theorem \ref{main} are satisfied. 
By \cite[Proposition 3.2]{CJ1}, we also have
${}^{\sigma}\blambda_M \in \mathcal{B}_{\theta}$. Therefore, there exists
${}^{\sigma} M \in  \mathrm{Irr}(K(\eta){\mathcal{H}})$ such that ${{}^\sigma \blambda_M}=\blambda_{{}^\sigma M } \in \mathcal{B}_{\theta}$. This action of $G$ on $\mathrm{Irr}(K(\eta){\mathcal{H}})$ agrees with the action on  $\mathrm{Irr}(K(\eta)\bar{\mathcal{H}})$, that is
$$[\mathrm{Res}_{\bar{\mathcal{H}}}^{\mathcal{H}}({}^\sigma M)]=
[\mathrm{Res}_{\bar{\mathcal{H}}}^{\mathcal{H}}(M)]=\sum_{N \in \bar{\omega}_{M}}[N].$$
 Let $L \in  \mathrm{Irr}(K(\eta){\mathcal{H}})$.
 The elements of $\bar{\omega}_{M}$ appear as composition factors in
 $\mathrm{Res}_{\bar{\mathcal{H}}}^{\mathcal{H}}(L)$ if and only if $L={}^gM$ for some
 $g \in G$.

By definition of $\mathcal{B}_{\theta}$, we get
$$ {\bf a}^{({\bf m},pr)}(M)= {\bf a}^{({\bf m},pr)}(\blambda_M)=
{\bf a}^{({\bf m},pr)}({}^\sigma \blambda_M)= {\bf a}^{({\bf m},pr)}({}^\sigma M)
$$
and
$$[V^{\blambda_M}]=[M] + \sum_{ {\bf a}^{({\bf m},pr)}(L)< {\bf a}^{({\bf m},pr)}(M)} [V^{\blambda_M} : L] [L].$$
By definition of the ${\bf a}$-function and Equation (\ref{aMN}), we get
$${\bf a}^{({\bf m},pr)}(M)={\bf a}^{({\bf m},pr)}(N),\text{ for all } N \in \bar{\omega}_M.$$
Moreover, if $L$ is a simple $K(\eta){\mathcal{H}}$-module such that $[V^{\blambda_M} : L] \neq 0$ and ${\bf a}^{({\bf m},pr)}(L)< {\bf a}^{({\bf m},pr)}(M)$, and $N' \in \mathrm{Irr}(K(\eta)\bar{{\mathcal{H}}})$
 is a composition factor of $\mathrm{Res}_{\bar{\mathcal{H}}}^{\mathcal{H}}(L)$, then
$${\bf a}^{({\bf m},pr)}(M) > {\bf a}^{({\bf m},pr)}(N').$$
We deduce that
$$
[\mathrm{Res}_{\bar{\mathcal{H}}}^{\mathcal{H}}(V^{\blambda_M})]=[\mathrm{Res}_{\bar{\mathcal{H}}}^{\mathcal{H}}(M)] +
\left(\begin{array}{c}
 \text{ sum of classes of simple modules with}\\ \text{{\bf a}-value strictly less than }
  {\bf a}^{({\bf m},pr)}(M)\end{array}\right),
$$whence we obtain
$$\sum_{E \in \bar{\Omega}_{\blambda_M}}[E]=\sum_{N \in \bar{\omega}_{M}}[N]+
\left(\begin{array}{c}
 \text{ sum of classes of simple modules with}\\ \text{{\bf a}-value strictly less than }
  {\bf a}^{({\bf m},pr)}(M)\end{array}\right).$$

Suppose that $N_M$ is a composition factor of $E^{\blambda_M}$. Then
${}^\sigma N_M$ is a composition factor of ${}^\sigma E^{\blambda_M}$, and, in general, ${}^g N_M$ is a composition factor of ${}^g E^{\blambda_M}$, for all $g \in G$. This is possible only if $|\bar{\omega}_{M}|=|\bar{\Omega}_{\blambda_M}|=|G_{\blambda_M}|$. For $g,\,h \in G_{\blambda_M}$, we get
$$[{}^g E^{\blambda_M}: {}^{h}N_M]=\left\{\begin{array}{ll}
 1, & \text{if }   g=h  \\ &\\
 0, & \text{otherwise.} \end{array}\right.$$
 Hence, we have
 $$[{}^g E^{\blambda_M}] =[{}^g N_M]+
\left(\begin{array}{c}
 \text{ sum of classes of simple modules with}\\ \text{{\bf a}-value strictly less than }
  {\bf a}^{({\bf m},pr)}({}^g N_M)\end{array}\right).$$
Thus, we have proved 
 the following result:

\begin{Th}\label{final}
The algebra $\bar{\mathcal{H}}$ admits a canonical basic set $\bar{\mathcal{B}}_\theta$ with respect to any specialisation  $\theta : \mathbb{Z}[\eta_l][q^{\pm 1}] \to K(\eta) $
such that $\theta (q)=\eta\in \mathbb{C}^*$, i.e., 
 there exists a set $\bar{\mathcal{B}}_{\theta}  \subset  \mathrm{Irr}(K(q)\bar{\mathcal{H}})$ with
$$| \bar{\mathcal{B}}_{\theta}   |=|\operatorname{Irr} (  K(\eta) \bar{ \mathcal{H}}|$$
such that the following property is satisfied: For  any  $N\in \operatorname{Irr} ( K(\eta)  \bar{\mathcal{H}})$, 
there exists a unique $E_N \in \bar{\mathcal{B}}_{\theta} $ such that  
  \begin{itemize}
\item $[E_N:N]=1$ \text{ and }  \smallbreak
\item  $[E:N]\neq 0$ for $E \in \mathrm{Irr}(K(q)\bar{\mathcal{H}})$ only if  ${\bf a}^{({\bf m},pr)}(E) > {\bf a}^{({\bf m},pr)}(E_N)$ or $E=E_N$.
\end{itemize}
In addition, we have that $E \in \bar{\mathcal{B}}_{\theta} $ if and only if there exists
 $\blambda\in \mathcal{B}_{\theta}\subset \Pi_n^l $ such that $E \in \bar{\Omega}_{\blambda}$.
\end{Th}

\begin{Rem}
In this section, we have also shown 
that the assumptions of \cite[Theorem 3.1]{GenJa}, which yields the existence of canonical basic sets for cyclotomic Hecke algebras of type $G(l,p,n)$, are satisfied for any choice of $\bar{\mathcal{H}}$.
\end{Rem}

\end{abs}

\begin{abs}
Let us give an example where we will apply Theorem \ref{final} in the case where $G(l,p,n)=G(3,3,2) \cong \mathfrak{S}_3$.\footnote{Of course, there is an easier way to deal with this case, but we simply want to illustrate the use of Theorem \ref{final} in a small example.}  Note that we have $d=1$, thus we can take ${\bf m}=(0,0,0)$. Let $r=2$ and consider the cyclotomic Hecke algebra 
$\mathcal{H}_{3,2}^{{\bf m},6}$ of type $G(3,3,2)$, with generators $t_1$, $t_2$ and relations
$$t_2t_1t_2=t_1t_2t_1,\,\,(t_1-q^6)(t_1+1)=(t_2-q^6)(t_2+1)=0.$$
The algebra $\mathcal{H}_{3,2}^{{\bf m},6}$ is a subalgebra of index $3$ of the cyclotomic Hecke algebra
$\mathcal{H}_{2}^{{\bf m},6}$ of type $G(3,1,2)$ with generators $T_0$, $T_1$ and relations
 $$T_0T_1T_0T_1=T_1T_0T_1T_0,\,\,T_0^3=1,\,\,\,
(T_1-q^6)(T_1+1)=0.$$
Let $\theta: \mathbb{Z}[\eta_3][q^{\pm 1}] \rightarrow \mathbb{Q}(\eta_{12})$ be a specialisation such that $\theta(q)=\eta_{12}$. Then the specialised Hecke algebra 
$\mathbb{Q}(\eta_{12})\mathcal{H}_{2}^{{\bf m},6}$ is generated by $T_0$ and $T_1$ with relations
$$T_0T_1T_0T_1=T_1T_0T_1T_0,\,\,T_0^3=1,\,\,\,
(T_1+1)^2=0.$$
By \cite[Theorem 1.1]{DiMa}, the specialised Hecke algebra 
$\mathbb{Q}(\eta_{12})\mathcal{H}_{2}^{{\bf m},6}$ is Morita equivalent to the algebra
$$\bigoplus_{n_1+n_2+n_3=2}\mathbb{Q}(\eta_{12})  {\bf H}_{n_1}^{{\bf m}^1,6}\otimes
\mathbb{Q}(\eta_{12})  {\bf H}_{n_2}^{{\bf m}^2,6}\otimes
\mathbb{Q}(\eta_{12})  {\bf H}_{n_3}^{{\bf m}^3,6},$$
where ${\bf m}^1={\bf m}^2={\bf m}^3=(0)$. Let $j \in \{1,2,3\}$.
The algebra $\mathbb{Q}(\eta_{12})  {\bf H}_{1}^{{\bf m}^j,6}$ is isomorphic to the algebra of the trivial group, and the algebra $\mathbb{Q}(\eta_{12})  {\bf H}_{2}^{{\bf m}^j,6}$ is isomorphic to the cyclotomic Hecke algebra of type $G(1,1,2) \cong \mathfrak{S}_2$ with quadratic relation $(T_1+1)^2=0$. Keeping the notation of \S\ref{s} and Proposition  \ref{conse}, we have
 $\Phi_{1}^1({\bf s}^j)=\{(1)\}$ and
 $\Phi_{2}^1({\bf s}^j)=\{(2)\}$.  Therefore, 
the canonical basic set  with respect to ${\theta}$  for $\mathcal{H}_{2}^{{\bf m},6}$ is
$$\mathcal{B}_\theta=\{ ((1),(1),\emptyset),\, (\emptyset,(1),(1)),\, ((1),\emptyset,(1)),\,
((2),\emptyset,\emptyset),\,(\emptyset,(2),\emptyset),\,(\emptyset,\emptyset,(2))\}.$$
Following Theorem \ref{final}, 
the canonical basic set  with respect to ${\theta}$  for $\mathcal{H}_{3,2}^{{\bf m},6}$ is
$$\bar{\mathcal{B}}_\theta=\{ E^{((1),(1),\emptyset)},\,E^{((2),\emptyset,\emptyset)}\}.$$

\end{abs}

%


\end{section}


\end{document}